\documentclass[twoside]{amsart}
\usepackage{amsmath,amssymb,amsthm}
\usepackage{verbatim}
\usepackage{paralist}
\usepackage{graphics} 
\usepackage{epsfig} 
\usepackage[colorlinks=true]{hyperref}
\hypersetup{urlcolor=blue, citecolor=red}

\newtheorem{theorem}{Theorem}[section]

\newtheorem{lemma}[theorem]{Lemma}
\newtheorem{proposition}{Proposition}

\theoremstyle{definition}
\newtheorem{definition}[theorem]{Definition}
\newtheorem{remark}{Remark}
\newcommand{\ep}{\varepsilon}

\newcommand{\B}{\mathcal{B}}
\newcommand{\D}{\mathcal{D}}
\newcommand{\E}{\mathcal{E}}
\newcommand{\F}{\mathcal{F}}

\newcommand{\R}{\mathbb{R}}
\newcommand{\N}{\mathbb{N}}
\newcommand{\dd }{\,\mathrm{d}}
\newcommand{\X}{X}

\newcommand{\EW}{\mathbb{E}}
\newcommand{\vX}{\widehat{\X}}

\title{Probabilistic Interpretation of the Calder\'{o}n Problem}
\author{Petteri Piiroinen}
\address{Petteri Piiroinen\\ÊDepartment of Mathematics and Statistics\\  University of Helsinki \\ \phantom{Ho}FI-00014 Helsinki, Finland}
\email{petteri.piiroinen@helsinki.fi}
\author{Martin Simon}

\address{Martin Simon\\ÊInstitute of Mathematics\\ Johannes Gutenberg University\\ 55099 Mainz, \phantom{Ho}Germany} 
\email{simon@math.uni-mainz.de}
\date{\today}

\chardef\bslchar=`\\

\providecommand{\qedsymbol}{\leavevmode
  \hbox to.77778em{%
  \hfil\vrule
  \vbox to.675em{\hrule width.6em\vfil\hrule}%
  \vrule\hfil}}

\catcode`\|=0
\begingroup \catcode`\>=13 
\gdef\?#1>{{\normalfont$\langle$\textup{#1}$\rangle$}}
\gdef\0{\relax}
\endgroup
\def\<#1>{{\normalfont$\langle$\textup{#1}$\rangle$}}

\def\latex/{{\protect\LaTeX}}

\begin{document}
\maketitle
\markboth{Petteri Piiroinen and Martin Simon}{Probabilistic Interpretation of the Calder\'{o}n Problem}
\begin{abstract}
In this paper, we use the theory of symmetric Dirichlet forms to 
give a probabilistic
interpretation of Calder\'{o}n's inverse conductivity problem in terms
of reflecting diffusion processes and their corresponding boundary trace
processes.
\end{abstract}

\section{Introduction}
Electrical impedance tomography (EIT) aims to reconstruct the unknown
conductivity $\kappa$ in the conductivity equation 
\begin{equation}\label{eqn:con}
\nabla\cdot(\kappa\nabla u)=0\quad\text{in }D
\end{equation}
from current and voltage measurements on the boundary of the domain $D$.
This inverse conductivity problem is known to be severely ill-posed,
that is, its solution is extremely sensitive with respect to measurement
and modeling errors. The uniqueness question, first proposed by Calder\'{o}n has been studied in  
extensively. And while it has been answered affirmatively for isotropic conductivities in the two-dimensional case by Astala and P\"aiv\"arinta \cite{AstalaP},
it is still unsettled, at least in full generality, in higher dimensions.  
In the previous work \cite{PiiSimon14} the authors have developed a probabilistic interpretation of the forward problem in form of a Feynman-Kac formula using reflecting diffusion processes.
In this work, we extend this probabilistic interpretation to Calder\'{o}n's inverse problem.
More precisely, we study the \emph{time-changed
process } $\widehat{\X}$ on $\partial D$ of
the reflecting diffusion process $\X$ with respect to the so-called \emph{local time on
the boundary}. For the special case of the reflecting Brownian motion on
the planar unit disk, the Beurling-Deny decomposition of this so-called
\emph{boundary process }is given by the
well-known \emph{Douglas integral}
\begin{equation*}
\int_D\lvert\nabla\mathcal{H}\phi(x)\rvert^2\dd x=\int_{\partial
D\times\partial D\backslash \delta}(\phi(\xi)-\phi(\eta))^2(4\pi
(1-\cos(\xi-\eta)))^{-1}\dd\xi\dd\eta,
\end{equation*} 
where $\mathcal{H}\phi$ denotes the harmonic function on $D$ with
Dirichlet boundary value $\phi$.
As a consequence, $\widehat{\X}$ is of pure jump type and its jumping
mechanism, the \emph{L\'{e}vy system}, is
governed by the \emph{Feller kernel } $(4\pi
(1-\cos(\xi-\eta)))^{-1}$. In this special case, $\widehat{\X}$ is the
\emph{symmetric Cauchy process }on the
unit circle, which leads somewhat naturally to the following
\emph{probabilistic inverse problem}: \emph{Is the reflecting Brownian motion the unique reflecting
diffusion process on the planar unit disk, whose boundary process is the
symmetric Cauchy process on the unit circle?} Clearly, this question can
be answered affirmatively for isotropic conductivities, however, the
authors are not aware of a \emph{probabilistic} proof.

The major goal of this work is to generalize the considerations from above in order to formulate
probabilistic inverse problems which are equivalent to Calder\'{o}n's
problem. We prove that the boundary process is of pure jump type and that its jumping measure is governed by the 
L\'{e}vy system of the boundary process. Moreover, we show that the latter is 
completely determined by the transition kernel density of the corresponding 
\emph{absorbing diffusion process} on $D$. We give explicit descriptions of both 
the L\'{e}vy system as well as the infinitesimal generator of the boundary process. 
This generalizes results by Hsu \cite{Hsu2} for the reflecting Brownian motion and 
enables the formulation of Calder\'{o}n's
problem in the form of equivalent probabilistic inverse problems in terms of the
\emph{excursion law} of the reflecting diffusion process and the L\'{e}vy system of the boundary process, respectively. 
We would like to stress the fact that in contrast to the reflecting Brownian motion, the reflecting diffusion process we study in this work
are in general not in the class of solutions to stochastic differential equations. In particular, we can not rely on It\^{o} calculus. 

The rest of this paper is structured as follows: We start in Section \ref{section2} by introducing our notation.
In Section \ref{section1a}, we recall both the forward problem of EIT as well as its probabilistic interpretation from \cite{PiiSimon14}.
Then, in Section \ref{section:BP}, we present two equivalent methods to define
the boundary process of a reflecting diffusion. The first one uses the
so-called \emph{trace Dirichlet form }and
its potential theory, whereas the second one is purely probabilistic
using time change with respect to the boundary local time. Moreover, we prove that the Dirichlet-to-Neumann map
is the \emph{infinitesimal generator }of
the boundary process. Finally, we study a L\'{e}vy system of the boundary process and the \emph{excursions} of
$\X$ which leads to the formulation of three equivalent probabilistic
inverse problems in Section \ref{section:probinv}. 

\section{Notation}
\label{section2}
Let $D$ denote a bounded Lipschitz domain in $\R^d$, $d\geq 2$, with
connected complement and Lipschitz parameters $(r_D,c_D)$, i.e., for
every $x\in\partial D$ we have after rotation and translation that
$\partial D\cap B(x,r_D)$ is the graph of a Lipschitz function in the
first $d-1$ coordinates with Lipschitz constant no larger than $c_D$ and
$D\cap B(x,r_D)$ lies above the graph of this function. Moreover, we set
$\R^d_-:=\{x\in\R^d: x\cdot \nu<0\}$, with $\nu=e_d$ the outward unit
normal on $\R^{d-1}$, where we identify the boundary of $\R^d_-$ with
$\R^{d-1}$, with straightforward abuse of notation. 

For Lipschitz domains, there exists a unique outward unit normal vector
$\nu$ a.e.\ on $\partial D$ so that the real Lebesgue spaces $L^p(D)$ and
$L^p(\partial D)$ can be defined in the standard manner with the usual
$L^p$ norms $\lvert\lvert\cdot\rvert\rvert_p$, $p=1,2,\infty$. The
standard $L^2$ inner-products are denoted by $\left<\cdot,\cdot\right>$
and $\left<\cdot,\cdot\right>_{\partial D}$, respectively. The
$d$-dimensional Lebesgue measure is denoted by $m$, the
$(d-1)$-dimensional Lebesgue surface measure is denoted by $\sigma$ and
$\lvert\cdot\rvert$ denotes the Euclidean norm on $\R^d$.

All functions in this work will be real-valued and derivatives are
understood in distributional sense. We use a diamond subscript to denote
subspaces of the standard Sobolev spaces containing functions with
vanishing mean and interpret integrals over $\partial D$ as dual
evaluations with a constant function, if necessary. For example, we will
frequently use the spaces
\begin{equation*}
H^{\pm 1/2}_{\diamond}(\partial D):=\Big\{\phi\in H^{\pm 1/2}(\partial
D):\left<\phi,1\right>_{\partial D}=0\Big\}
\end{equation*}
and
\begin{equation*}
H^{1}_{\diamond}(D):=\Big\{\phi\in H^{1}(D):\left<\phi,1\right>=0\Big\}.
\end{equation*}

Moreover, we will frequently assume that $\partial D$ is partitioned
into two disjoint parts, $\partial_1 D$ and $\partial_2 D$. We denote by
$H_0^1(D\cup\partial_1 D)$ the closure of $C_c^{\infty}(D\cup\partial_1
D)$, the linear subspace of $C^{\infty}(\overline{D})$ consisting of
functions $\phi$ such that $\mathrm{supp}(\phi)$ is a compact subset of
$D\cup\partial_1 D$, in $H^1(D)$.
For the reason of notational compactness, we use the Iverson brackets:
Let $S$ be a mathematical statement, then 
\begin{equation*}\left[S\right]=\begin{cases}
1,\quad&\text{if }S\text{ is true}\\
0,\quad&\text{otherwise}.
\end{cases}
\end{equation*}
We also use the Iverson brackets $[x\in B]$ to denote the indicator
function of a set $B$, which we abbreviate by $[B]$ if there is no
danger of confusion. 

In what follows, all unimportant constants are denoted $c$, sometimes
with additional subscripts, and they may vary from line to line. 


\section{The forward problem and its probabilistic interpretation}\label{section1a}

We assume that the, possibly anisotropic, conductivity is defined by a symmetric,
matrix-valued function $\kappa:D\rightarrow \mathbb{R}^{d\times d}$ with
components in $L^{\infty}(D)$ such that $\kappa$ is uniformly bounded
and uniformly elliptic, i.e., there exists some constant $c>0$ such that
\begin{equation}\label{eqn:ellipticity}
c^{-1}\lvert\xi\rvert^2\leq \xi\cdot\kappa (x)\xi\leq
c\lvert\xi\rvert^2,\quad \text{for every
}\xi\in\mathbb{R}^d\text{ and a.e. }x\in D.  
\end{equation}

The forward problem of electrical impedance tomography can be described
by different measurement models. In the so-called \textit{continuum
model}, the conductivity equation (\ref{eqn:con})
is equipped with a co-normal boundary condition 
\begin{equation}\label{eqn:continuum}
\partial_{\kappa\nu}u:= \kappa\nu\cdot\nabla u\vert_{\partial D}=f\quad\text{on }\partial D,
\end{equation}
where $f$ is a measurable function modeling the signed density of the
outgoing current. The boundary value problem (\ref{eqn:con}),
(\ref{eqn:continuum}) has a solution if and only if 
\begin{equation}\label{eqn:compatibility}
\left<f,1\right>_{\partial D}=0.
\end{equation}
Physically speaking, this means that the current must be conserved.
Given an appropriate function $f$, the solution to (\ref{eqn:con}),
(\ref{eqn:continuum}) is unique up to an additive constant, which
physically corresponds to the choice of the ground level of the
potential. If $f\in H^{-1/2}_{\diamond}$, then there exists a unique
equivalence class of functions $u\in H^1(D)/\R$ that satisfies the weak
formulation of the boundary value problem
\begin{equation*}
\int_D\kappa\nabla u\cdot\nabla v\dd x=\left <f, v\vert_{\partial D}\right
>_{\partial D}\quad \text{for all }v\in H^1(D)/\R,
\end{equation*} 
where $v\vert_{\partial D}:=\gamma v$ and $\gamma:H^1(D)/\R\rightarrow
H^{1/2}(\partial D)/\R=(H^{-1/2}_{\diamond}(\partial D))'$ is the
standard trace operator. Note that we occasionally write $v$ instead of
$v\vert_{\partial D}$ for the sake of readability.

In his seminal paper \cite{Fukushima}, Fukushima established a
one-to-one correspondence between regular symmetric Dirichlet forms and
symmetric Hunt processes, which is the foundation for the construction
of stochastic processes via Dirichlet form techniques. Therefore we
assume that the reader is familiar with the theory of symmetric
Dirichlet forms, as elaborated for instance in the monograph
\cite{Fukushimaetal}.

Let us consider the following symmetric bilinear forms on $L^2(D)$:
\begin{equation}\label{eqn:Dirichlet}
\E(v,w):=\int_{D}\kappa\nabla v(x)\cdot\nabla w(x)\dd  x,\quad
v,w\in\D(\E):=H^1(D)
\end{equation}
and for the particular case $\kappa\equiv 1/2$, which is of special
importance, we set
\begin{equation}\label{eqn:Dirichlet1}
\E^{\text{BM}}(v,w):=\frac{1}{2}\int_{D}\nabla v(x)\cdot\nabla w(x)\dd  x,\quad
v,w\in\D(\E^{\text{BM}}):=H^1(D).
\end{equation} 

The pair $(\E, \D(\E))$ defined by (\ref{eqn:Dirichlet}) is a strongly
local, regular symmetric Dirichlet
form on $L^2(D)$. In particular, there exist an $\E$-exceptional set
$\mathcal{N}\subset\overline{D}$ and a conservative diffusion process
$\X=(\Omega,\F,\{\X_t,t\geq 0\},\mathbb{P}_{x})$, starting from 
$x\in\overline{D}\backslash \mathcal{N}$ such that 
$\X$ is associated with $(\E,\D(\E))$. Without loss of generality let us
assume that $\X$ is defined on the \emph{canonical
sample space }$\Omega=C([0,\infty);\overline{D})$. 
It is well-known that the symmetric Hunt process associated with
(\ref{eqn:Dirichlet1}) is the reflecting Brownian motion. Therefore, we
call the symmetric Hunt process associated with (\ref{eqn:Dirichlet}) a
\emph{reflecting diffusion process}. 

Let us briefly recall the concept of the \emph{boundary local time }of
reflecting diffusion processes, see, e.g.,
\cite{Pardoux,Hsu,PiiSimon14}. If the diffusion
process is the solution to a stochastic differential equation, say the
reflecting Brownian motion, then the boundary local time is given by the
one-dimensional process $L$ in the Skorohod decomposition, which
prevents the sample paths from leaving $\overline{D}$, i.e.,
\begin{equation}\label{eqn:SkorohodRBM}
\X_t=x+W_t-\frac{1}{2}\int_0^t\nu(\X_s)\dd L_s, 
\end{equation} 
$\mathbb{P}_x$-a.s.\ for q.e.\ $x\in\overline{D}$. This boundary local
time is a continuous non-decreasing process which increases only when
$X_t\in\partial D$, namely for all $t\geq 0$ and q.e.\ $x\in\overline{D}$
\begin{equation*}
L_t=\int_0^t[\partial D](\X_s)\dd L_s,
\end{equation*}
$\mathbb{P}_x$-a.s.\ and
\begin{equation*}
\EW_x\int_0^t[\partial D](\X_s)\dd s=0. 
\end{equation*}
Although the reflecting diffusion process associated with
(\ref{eqn:Dirichlet}) does in general not admit a Skorohod decomposition
of the form (\ref{eqn:SkorohodRBM}), we may still define a continuous
one-dimensional process with these properties. More precisely, by the
Lipschitz property of $\partial D$, we have that $D\cap
B(x,r_D)=\{(\tilde{x},x_d):x_d>\gamma(\tilde{x})\}\cap B(x,r_D)$ and the
Lipschitz function $\gamma$ is differentiable a.e.\ with a bounded
gradient. In particular, we have for every Borel set $B\subset \partial
D\cap B(x,r_D)$ that
\begin{equation*}
\sigma(B)=\int_{\{\tilde{x}:(\tilde{x},\gamma(\tilde{x}))\in
B\}}\Big(1+\lvert \nabla
\gamma(\tilde{x})\rvert^2\Big)^{1/2}\dd\tilde{x}
\end{equation*}
and a straightforward computation yields that the Lebesgue surface
measure $\sigma$ is a smooth measure with respect to $(\E,\D(\E))$
having finite energy, i.e., 
\begin{equation*}
\int_{\partial D}\lvert v\rvert \dd\sigma(x)\leq c\lvert\lvert
v\rvert\rvert_{\E_1}\quad \text{for all }v\in \D(\E)\cap
C(\overline{D}),
\end{equation*}
where we have used the inner product
$\E_1(\cdot,\cdot):=\E(\cdot,\cdot)+\left<\cdot,\cdot\right>$.
\begin{definition}
The positive continuous additive functional of $X$ whose Revuz measure
is given by the Lebesgue surface measure $\sigma$ on $\partial D$, i.e.,
the unique $L \in\mathcal{A}_c^+$ such that
\begin{equation}\label{eqn:Revuz}
\lim_{t\rightarrow 0+}\frac{1}{t}\int_D\EW_x\Big\{\int_0^t\phi(\X_s)\dd
L_s\Big\}\psi(x)\dd x= \int_{\partial D}\phi(x)\psi(x)\dd\sigma(x)
\end{equation}
for all non-negative Borel functions $\phi$ and all $\alpha$-excessive
functions $\psi$, is called the \emph{boundary local time} of the
reflecting diffusion process $\X$.
\end{definition}

It has been shown in the recent work \cite{PiiSimon14} that the $\E$-exceptional
set $\mathcal{N}$ is actually empty. 

\begin{proposition}[{\cite[Proposition 1]{PiiSimon14}}]\label{thm:1}
$p\in C^{0,\delta}((0,T]\times\overline{D}\times\overline{D})$ for some
$\delta\in(0,1)$, i.e., for each fixed $0<t_0\leq T$, there exists a
positive constant $c$ such that 
\begin{equation}\label{eqn:Hold}
\lvert p(t_2,x_2,y_2)-p(t_1,x_1,y_1)\rvert\leq
c(\sqrt{t_2-t_1}+\lvert x_2-x_1\rvert+\lvert y_2-y_1\rvert)^{\delta}
\end{equation}
for all $t_0\leq t_1\leq t_2\leq T$ and all
  $(x_1,y_1),(x_2,y_2)\in\overline{D}\times\overline{D}$.
Moreover, the mapping $t\mapsto p(t,\cdot,\cdot)$ is analytic from
$(0,\infty)$ to $C^{0,\delta}(\overline{D}\times\overline{D})$.  
\end{proposition}

By \cite[Theorem 2]{Fukushima2}, the existence of a H\"older continuous
transition kernel density ensures that we may refine the process $\X$ to
start from every $x\in\overline{D}$ by identifying the strongly
continuous semigroup $\{T_t,t\geq 0\}$ with the transition semigroup
$\{P_t,t\geq 0\}$. In particular, if $v$ is continuous and locally in
$H^1(D)$, the Fukushima decomposition holds for every
$x\in\overline{D}$, i.e., 
\begin{equation}\label{eqn:Fukudecomp}
v(\X_t)=v(\X_0)+M_t^{v}+N_t^{v},\quad\text{for all }
t> 0,
\end{equation}
$\mathbb{P}_{x}$-a.s.,
where $M^{v}$ is a martingale additive functional of $\X$ having finite
energy and $N^{v}$ is a continuous additive functional of $\X$ having
zero energy. 

Moreover, both $M^{v}$ and $N^{v}$ can be taken to be additive
functionals of $\X$ in the strict sense, cf.
\cite[Theorem 5.2.5]{Fukushimaetal}. 

Finally, note that the $1$-potential of the Lebesgue surface measure
$\sigma$ of $\partial D$ is the solution to an elliptic boundary value
problem on a Lipschitz domain with bounded data. By elliptic regularity
theory, cf., e.g.,  \cite{Grisvard}, this solution is continuous,
implying that the boundary local time $L$
exists as a positive continuous additive functional in the strict sense,
cf.~\cite[Theorem 5.1.6]{Fukushimaetal}. 

For the probabilistic interpretation of the Neumann boundary value
problem corresponding to the continuum model, we require the
following assumption on the conductivity $\kappa$:
\begin{enumerate}
\item[(A1)]
There exists a neighborhood $\mathcal{U}$ of the boundary $\partial D$
such that $\kappa\vert_{\mathcal{U}}$ is isotropic and equal to $1$. 
\item[(A2)]
There exists a neighborhood $\mathcal{U}$ of the boundary $\partial D$
such that $\kappa\vert_{\mathcal{U}}$ is H\"older continuous.
\end{enumerate}
\begin{remark}
Notice that, theoretically, the assumption (A1) imposes no restriction to generality. 
More precisely, it can be shown using extension techniques that for domains $\widehat{D}, D\subset
\R^d$ such that $D\subset\widehat{D}$, the knowledge of both, the
Dirichlet-to-Neumann map $\Lambda_{\kappa}$ on $\partial D$ and
$\kappa\vert_{\widehat{D}\backslash\overline{D}}$ yields the
Dirichlet-to-Neumann map $\hat{\Lambda}_{\kappa}$ on $\partial \widehat{D}$. 
\end{remark}

The main result for the continuum model (\ref{eqn:con}),
(\ref{eqn:continuum}) is the following theorem from \cite{PiiSimon14}.
\begin{theorem}[{ \cite[Theorem 6.4.]{PiiSimon14}}]\label{thm:cont}
Assume that $\kappa$ satisfies (A1) or (A2) and let $D$ denote a bounded Lipschitz domain. Let $f$ be a bounded Borel function satisfying
$\left<f,1\right>_{\partial D}=0$.  Then there is a unique weak solution
$u\in C(\overline{D})\cap H^1_{\diamond}(D)$ to the boundary value
problem (\ref{eqn:con}), (\ref{eqn:continuum}). This solution admits the
Feynman-Kac representation
\begin{equation}
u(x)=\lim_{t\rightarrow\infty}\EW_{x}\int_0^t f(\X_s)\dd 
L_s\quad\text{for all }x\in\overline{D}.
\end{equation}  
\end{theorem}

\begin{remark}\label{rem:1}
We would like to emphasize that the regularization technique employed in the proof of
\cite[Theorem 6.4.]{PiiSimon14} may be easily modified to prove the Feynman-Kac
formula 
\begin{equation*}
u(x)=\EW_x\phi(\X_{\tau(D)}),\quad x\in D
\end{equation*}
for the conductivity equation (\ref{eqn:con}) with Dirichlet boundary
condition $u\vert_{\partial D}=\phi,$ where $\phi\in H^{1/2}(D)$ and
$$\tau(D):=\inf\{t\geq 0:\X_t\in\R^d\backslash D\}$$ denotes the
\emph{first exit time} from the domain $D$. 
This follows from the Lipschitz property of $\partial D$,
implying that all points of $\partial D$ are
\emph{regular} in the sense of \cite[Chapter
4.2]{Karatzas}.
\end{remark}

\section{Boundary processes of reflecting diffusions}\label{section:BP}
\subsection{Definition and properties}
As we are going to use regularity results that are not readily available
for general Lipschitz domains, we assume  throughout this chapter that
$D$ has a smooth boundary in order to avoid technical difficulties. However, we expect our results to hold for general Lipschitz domains.

Let $\X^0$ denote the \emph{absorbing diffusion process }on $D$ which is obtained from the
reflecting diffusion process $\X$ on $\overline{D}$ by killing upon
hitting of $\R^d\backslash D$, i.e., 
\begin{equation}\X^0_t := \begin{cases}
\X_t,\quad&t\leq\tau(D)\\
\partial,\quad& t>\tau(D).
\end{cases}
\end{equation}  
By the Markov property of $\X$, 
$\X^0$ possesses a transition kernel density which may be expressed as
\begin{equation}\label{eqn:dens_dir}
p^0(t,x,y)= p(t,x,y)-\EW_x\big\{p(t-\tau(D),\X_{\tau(D)},y)[\tau(D)< t]\big\}
\end{equation} 
and the regular symmetric Dirichlet form on $L^2(D)$ associated with $\X^0$ is $(\E,H^1_0(D))$, cf. \cite{Fukushimaetal}.

\begin{lemma}\label{killed:diff}
Let $\kappa:\overline{D}\rightarrow \R^{d\times d}$ be a symmetric,
uniformly bounded and uniformly elliptic conductivity with components
$\kappa_{ij}\in C^{0,1}(\overline{D})$, $i,j=1,...,d$, such that
$\kappa$ satisfies (A1). Then the absorbing diffusion process $\X^0$
possesses a transition kernel density $p^0$ with the following
properties:
\begin{enumerate}
\item[$(i)$] $p^0$ is jointly H\"older continuous with respect to
  $(t,x,y)$; 
\item[$(ii)$]  $p^0 $ is in $H_0^1(D)$ as a function of $x$ and $y$,
respectively; 
\item[$(iii)$] $p^0\vert_{\mathcal{U}}\in C^1(\mathcal{U}\cup\partial D)\cap
C^2(\mathcal{U})$ as a function of $x$ and $y$, respectively, where
$\mathcal{U}$ denotes the neighborhood of $\partial D$ from assumption
(A1);
\item[(iv)] $p^0$ is continuously differentiable with respect to $t$ as
a Banach space valued map.
\end{enumerate}
\end{lemma}
\begin{proof}
The property $(i)$ follows directly from Proposition~\ref{thm:1} by (\ref{eqn:dens_dir}). 

To show the property $(ii)$, first note that for every $x \in D$, we
have $v_{x,t}(y) = p^0(t/2, y, x)$ is H\"older continuous for every $t > 0$
and in particular it is bounded and in $L^2(D)$ since $D$ is bounded.
This implies that for every $t > 0$ we know that
\[
T_{t/2} v_x(y) = \E_y v_x(\X_{t/2}) [\, t/2 < \tau_D \,] = \int_D p^0(t/2,y,z)
p^0(t/2,y,x) \dd z = v_{x,2t}(y)
\]
and moreover, $T_{t/2} v_x \in \D(\mathcal L) = H^1_0(D)$. This implies
the claim.

The property $(iii)$ is a direct consequence of assumption (A1) and the regularity of
the fundamental solution for the heat equation, cf. \cite{Hsu}.

Finally, $(iv)$ follows by the same reasoning we used in the proof of
Proposition \ref{thm:1}, see \cite{PiiSimon14}.
\end{proof}

A function $u\in H^1_{\text{loc}}(D)$ is said to be \emph{$\mathcal{L}$-harmonic}, provided 
\begin{equation}
\EW_x \lvert u(\X_{\tau(\widehat{D})})\rvert<\infty\quad\text{and}\quad u(x)=\EW_x u(\X_{\tau(\widehat{D})})\quad \text{for all }x\in \widehat{D}
\end{equation}
and every relatively compact open subset $\widehat{D}\subset D$, which is clearly equivalent to 
\begin{equation}\label{eqn:test}
\int_D\kappa\nabla u\cdot \nabla v\dd x= 0\quad\text{for all }v\in C_c(D),
\end{equation}
cf. Remark \ref{rem:1}. By the same remark, the \emph{$\mathcal{L}$-harmonic extension operator }$\mathcal{H}: L^{\infty}(\partial D)\cap H^{1/2-\ep}(\partial D)\rightarrow H_{\text{loc}}^1(D)$, $\ep\in [0,1/2)$, defined by 
\begin{equation*}
\mathcal{H}\phi(x):=\EW_x\phi(\X_{\tau(D)}),\quad x\in D
\end{equation*}
is well-defined. If $\ep\in(0,1/2)$, this is due to the fact that the Dirichlet problem
\begin{equation}\label{eqn:Dirichletprob}
\nabla\cdot(\kappa\nabla u)=0\quad\text{in }D,\quad u\vert_{\partial D}=\phi,
\end{equation}
admits a unique solution $u\in H^{1-\ep}(D)\cap H^1_{\text{loc}}(D)$ satisfying (\ref{eqn:test}), which follows from a saddle point formulation introduced in \cite{Necas}.

The following lemma generalizes a result from Aizenman and Simon
\cite{AizenmanSimon} for the absorbing Brownian motion.
\begin{lemma}
 \label{Poissonkernel}
Let $\kappa:\overline{D}\rightarrow \R^{d\times d}$ be a symmetric,
uniformly bounded and uniformly elliptic conductivity 
and suppose, in addition, $\kappa$ satisfies (A2). Then 
for every bounded Borel function $\phi$ on $\partial D$ we have
 \[
 \EW_x \phi(\X_{\tau(D)}) [\,\tau(D)\leq t\,] = \int_{\partial D}
 \int_{0}^t\phi(y)\partial_{\kappa\nu(y)} p^0(s,x,y)\dd  s \dd
 \sigma(y),\quad x\in D.
 \]
\end{lemma}
\begin{proof}
We will proceed in the spirit of the Aizenman and
Simon~\cite{AizenmanSimon} proof but we will reformulate it to cover our case. 
Let $\psi \in C_c^\infty([0,\infty))$ be a smooth approximation of the
indicator function $[ 0 \le s < t]$ and $v_0(y) =
p^0(\epsilon, x, y)$ for given $x \in D$ and $\epsilon > 0$. Note that
$v_0 \in L^2(D)$.

Let $w \in H^1(D)$. Then since $v_s = T_sv_0$ satisfies an abstract Cauchy
problem, cf., e.g., \cite{Pazy}, we have by the equivalent variational
formulation that
\[
\int_0^\infty
\langle \mathcal L v_s, w \rangle \psi(s) \dd s 
=
\int_0^\infty \langle v_s, w \rangle \dot\psi(s) \dd s
+ \EW_x w(\X_\epsilon) [\, \epsilon < \tau \,]
\]
Let $\varphi_\epsilon$ be a non-negative $H^1_0(D)$ function such that
is an approximation of the indicator function $[D_\epsilon]$ or more
precisely
$\varphi_\epsilon[D_{2\epsilon}] = [D_{2\epsilon}]$ and
$\varphi_\epsilon[D \setminus D_\epsilon] \equiv 0$. Rewriting $w$ on
the right-hand side by $w = w\varphi_\epsilon + w(1-\varphi_\epsilon)$
we get by the definition of weak derivative that
\[
\begin{split}
\langle \mathcal Lv_s, w \rangle
&= -\langle \kappa \nabla v_s, \nabla (\varphi_\epsilon w)
\rangle + \langle \mathcal Lv_s, (1-\varphi_\epsilon) w\rangle\\
&= -\langle \kappa \nabla v_s, \nabla w \rangle
+ \langle \kappa \nabla v_s, \nabla ((1-\varphi_\epsilon) w)\rangle
+ \langle \mathcal Lv_s, (1-\varphi_\epsilon) w\rangle
\end{split}
\]
If $w$ is an $\mathcal L$-harmonic, the first term on the right is $0$.
If $\kappa$ is H\"older continuous close to the boundary $\partial D$,
then the elliptic regularity shows that
\[
\lim_{\epsilon \to 0}\lvert \langle \mathcal Lv_s, (1-\varphi_\epsilon) w\rangle \rvert
= 0
\]
since $(1-\varphi_\epsilon)w \to 0$ in $H^{1-\delta}(D)$. The term
\[
\langle \kappa \nabla v_s, \nabla ((1-\varphi_\epsilon) w)\rangle
= 
\langle \kappa \nabla v_s,(1-\varphi_\epsilon) \nabla w) \rangle
- \langle \kappa \nabla v_s,w \nabla\varphi_\epsilon \rangle.
\]
Since $w \in H^1(D)$, we have
\[
\lim_{\epsilon \to 0}\lvert \langle \kappa \nabla v_s,
(1-\varphi_\epsilon) \nabla w\rangle \rvert
= 0
\]
since $(1-\varphi_\epsilon)\nabla w \to 0$ in $L^2(D)$. Since the boundary
$\partial D$ is Lipschitz and compact, it can be further divided into a
finite partition of $H^1(D)$ and within this partition the $\nabla
\varphi_\epsilon(x) = -\epsilon^{-1}[ x \in D_\epsilon \setminus
D_{2\epsilon} ] \nu(x^*)$ where $x^* \in \partial D$ is the unique point
such that after the Lipschitz change of coordinates, $x =
\beta\nu(x^*)$.
Therefore,
\[
-\lim_{\epsilon \to 0}\lvert \langle \kappa \nabla v_s,
w \nabla \varphi_\epsilon\rangle \rvert
= \int_{\partial D} \partial_{\kappa \nu(y)} v_s(y) w(y) \dd \sigma(y)
\]
and all in all, we have shown that
\[
\begin{split}
\int_0^\infty \psi(s)
\int_{\partial D} \partial_{\kappa \nu(y)} v_s(y) w(y) \dd \sigma(y)
\dd s
=&
\int_0^\infty \langle v_s, w \rangle \dot\psi(s) \dd s
+ \EW_x w(\X_\epsilon) [\, \epsilon < \tau \,] \\
&+ o(1).
\end{split}
\]
Since the semigroup $\{T_t,t\geq 0\}$ is continuous and $\psi$ is an approximation of the
indicator function $[0 \le s < t]$, we further deduce that
\[
\int_0^t \int_{\partial D} \partial_{\kappa \nu(y)} v_s(y) w(y) \dd
\sigma(y) \dd s
= -\langle v_t, w \rangle + w(x) + o(1)
\]
for $m$-a.e.\ $x \in D$. Since 
\[
\langle v_t, w \rangle = \int_D v_t(y) w(y) \dd y
= \int_D \int_D \, p^0(t,y,z)p^0(\epsilon,x,z) w(y)\dd y\dd z
= T_{t+\epsilon} w(x)
\]
we have deduced that
\[
\int_0^t \int_{\partial D} \partial_{\kappa \nu(y)} v_s(y) w(y) \dd
\sigma(y) \dd s
= -T_t w(x)+ w(x) + o(1)
\]
for $m$-a.e.\ $x \in D$ and for every $\mathcal L$-harmonic function $w
\in H^1(D)$. Using the Chapman--Kolmogorov and the analyticity of
$p^0$ with respect to $t$ we can write this as
\[
\int_{\epsilon}^{t+\epsilon} \int_{\partial D} \partial_{\kappa \nu(y)}
p^0(s,x,y) w(y) \dd
\sigma(y) \dd s
= -T_t w(x)+ w(x) + o(1)
\]
which implies that
\begin{equation}\label{eqn:aizsim}
\int_0^t \int_{\partial D} \partial_{\kappa \nu(y)}
p^0(s,x,y) w(y) \dd
\sigma(y) \dd s
= -T_t w(x)+ w(x)
\end{equation}
for $m$-a.e.\ $x \in D$. We can use the Markov property (as in the proof of
of \cite[Lemma~6.3]{PiiSimon14}) to show that identity~\eqref{eqn:aizsim} holds
for every $x \in D$ and for every $\mathcal L$-harmonic function $w \in
H^1(D)$. In particular, we can reformulate this as
\begin{equation}\label{eqn:aizsim3}
\begin{split}
\int_0^t \int_{\partial D} \partial_{\kappa \nu(y)}
p^0(s,x,y) \phi(y) \dd
\sigma(y) \dd s
&= -T_t \mathcal H\phi(x)+ \mathcal H\phi(x)\\
&= -\EW_x \mathcal H\phi(\X_t) [\, t < \tau \,] + \mathcal H\phi(x)
\end{split}
\end{equation}
for every $\phi \in H^{1/2}(\partial D)$. By Markov property and the
fact that for every $x \in D$ the stopping time $\tau > 0$ a.s.\ we have
\[
\EW_x\mathcal H\phi(\X_t)) [\, t < \tau \,] = \EW_x \EW (
\phi(X_\tau) [ \, t < \tau \,] \,\vert\, \mathcal F_t ) = \EW_x
\phi(X_\tau) [ \, t < \tau \,].
\]
Thus, the right-hand side of the identity~\eqref{eqn:aizsim3} is by the
definition of the harmonic extension operator
\[
-\EW_x \phi(X_\tau) [ \, t < \tau \,] + \EW_x \phi(X_\tau)
= \EW_x \phi(X_\tau) [ \, t \ge \tau \,].
\]
We have therefore shown that 
\begin{equation}\label{eqn:aizsim2}
\begin{split}
\int_0^t \int_{\partial D} \partial_{\kappa \nu(y)}
p^0(s,x,y) \phi(y) \dd
\sigma(y) \dd s
&= \EW_x \phi(X_\tau) [ \, \tau \le t \,]
\end{split}
\end{equation}
for every $\phi \in H^{1/2}(\partial D)$. This implies by the Monotone
Class Theorem that the identity~\eqref{eqn:aizsim2} holds for every
bounded and measurable function $\phi$.
\end{proof}

Lemma \ref{Poissonkernel} yields the joint distribution of the pair
$(\tau(D),\X_{\tau(D)})$ with respect to the measure $\mathbb{P}_x$,
$x\in \overline{D}$, namely
\begin{equation}
\mathbb{P}_x\{\tau(D)\in\mathrm{d}t,\X_{\tau(D)}\in \dd
y\}=\partial_{\kappa\nu(y)}p^0(t,x,y)\dd t\dd\sigma(y).
\end{equation}
In particular, there exists a \emph{Poisson kernel }with respect to the Lebesgue surface measure $\sigma$ which is
given by
\begin{equation*}
K_{\kappa}(x,y):=\int_0^{\infty}\partial_{\kappa\nu(y)}p^0(t,x,y)\dd
t,\quad x\in \overline{D},
\end{equation*} 
such that the hitting probability of $B\in \B(\partial D)$ for $\X$ starting in $x\in \overline{D}$ is 
\begin{equation*}
\mathbb{P}_x\{\X_{\tau(D)}\in B\}=\int_BK_{\kappa}(x,y)\dd\sigma(y).
\end{equation*} 

Following \cite{Chen2}, the so-called \emph{trace Dirichlet form }
$(\widehat{\E},\D(\widehat{\E}))$ of
$(\E,\D(\E))$ is given by 
\begin{equation}
\widehat{\E}(v,w):=\E(\mathcal{H}
v,\mathcal{H}w),\quad
\D(\widehat{\E}):=\D_e(\E)\vert_{\partial D}\cap L^2(\partial D).
\end{equation}
It is a symmetric regular Dirichlet form on $L^2(\partial D)$, cf.
\cite[Theorem 6.2.1]{Fukushimaetal}, whose domain is characterized by
the following lemma.
\begin{lemma}
Let $\kappa:\overline{D}\rightarrow \R^{d\times d}$ be a symmetric,
uniformly bounded and uniformly elliptic conductivity.
Then $\D(\widehat{\E})=H^{1/2}(\partial D)$.
\end{lemma}
\begin{proof}
By the standard trace theorem it suffices to show $\D_e(\E) \subset H^1(D) =
\D(\E)$. Let
$v\in \D_e(\E)$. Therefore, 
by the definition of the extended Dirichlet space,
there is a sequence $\{v_k\}_{k\in\N} \subset H^1(D)$ such that
$(v_k)$ converges $m$-a.e.\ on $D$ to $v$ and $(v_k)$ is $\E$-Cauchy
sequence.

Since $(v_k)$ is a $\E$-Cauchy, we have by the Poincar\'{e} inequality
and the boundedness of $\kappa$ that
\begin{equation}
\lVert v_n-v_m - Mv_n + Mv_m \rVert^2_{H^1(D)}\leq c\lVert \nabla (v_n-
v_m)\rVert^2_2\leq c_1{\E(v_n-v_m,v_n-v_m)},
\end{equation}
where $M u :=\lvert D\rvert^{-1}\langle u, 1\rangle_2[D]$. In other words, $\{v_n -
M v_n\}_{n\in \N}$ is an $H^1(D)$-Cauchy sequence
and thus exists a $w\in H^1(D)$ such that $v_n-Mv_n \to w$ in $H^1(D)$.
Furthermore, we may choose a subsequence $\{w_n - Mw_n\} \subset
\{v_n-Mv_n\}$ such that $w_n - Mw_n \to w$ for $m$-a.e.\ in $D$. On the other
hand, we have that $w_n \to v$ for $m$-a.e.\ in $D$ and thus, $Mw_n = w_n - (w_n -
Mw_n) \to w - v$ for $m$-a.e.\ in $D$. Since $Mw_n$ is a constant
function for every $n$, we deduce that $w-v = c[D]$. Therefore, $v = w +
c[D] \in H^1(D)$ as claimed.
\end{proof}
Now let us consider the Hunt process $\vX=(\widehat{\Omega},\widehat{\F},\{\widehat{\X},t\geq 0\},\widehat{\mathbb{P}}_x)$ associated with the trace Dirichlet form, which is a $\sigma$-symmetric pure jump process on $\partial D$. More precisely, as the energy measure of $M^v$, $v\in\D_e(\E)=H^1(D)$ does not charge the boundary $\partial D$ and due to the existence of the Poisson kernel, we may apply \cite[Theorem 6.2]{Chen2} to obtain for $ v,w\in \D(\widehat{\E})\cap C_c(\partial D)$  the Beurling-Deny decomposition 
\begin{equation}
\widehat{\E}(v,w)=\int_{\partial D\times\partial D\backslash\delta}(v(x)-v(y))(w(x)-w(y))\dd\widehat{\mathcal{J}}(x,y).
\end{equation}
The \emph{jumping measure }$\dd\widehat{\mathcal{J}}(x,y)$ is determined by the absorbing diffusion process $\X_0$ and can be characterized by the so-called \emph{Feller measure} which depends only on the associated symmetric strongly continuous contraction semigroup on $L^2(D)$ and the Poisson kernel, cf. \cite{Chen2}. $\vX$ will be called the \emph{boundary process }of $\X$.

Note that the boundary process of $\X$ may be equivalently constructed
using the general theory of time changes of diffusion processes with
respect to positive continuous additive functionals:  We have seen that
the boundary local time $L$ is a nondecreasing, $\{\F_t,t\geq
0\}$-adapted process
that increases only when $\X$ is at the boundary. 
Following~\cite{Fukushimaetal}, 
we define the right-continuous
right-inverse $\tau$ of $L$ by 
\begin{equation}
  \label{RCRI}
  \tau(s) := \sup\{ r \ge 0 : L_r \le s \}.
\end{equation}
The random variable $\tau(s)$, $s \in [0,\infty)$, is a stopping time with
respect to the right-continuous history $\{\F_t,t\geq 0\}$ of $\X$ since
$\{ \tau(s) \ge t \} = \{ L_t \le s \} \in \mathcal F_t$ and moreover,
by continuity of $\X$ we see that for every $s \in [0,\infty)$ the
process $\X$ is at the boundary $\partial D$ at time $\tau(s)$.
Therefore, we can equivalently define the boundary process $\vX$ of $\X$
as the time-changed trace
  \[
  \vX_t := \X_{\tau(t)}
  \]
  and the \emph{boundary filtration}
  \[
  \widehat{\mathcal F}_t := \mathcal F_{\tau(t)}.
  \]

To be precise, we know that the boundary local time $L$ is a positive
continuous additive functional in the strict sense since the
corresponding 1-potential is bounded by elliptic regularity
and therefore, the boundary $\partial D$
coincides with the so-called \emph{quasi-support} of $L$,
cf.~\cite[Theorem 5.1.5]{Fukushimaetal}. Moreover, since $\partial D$ is
smooth, every boundary point is a
\emph{regular} point, cf. \cite{Karatzas}, so that the boundary process
is a $\sigma$-symmetric Hunt process on the boundary $\partial D$, cf.
\cite[Theorem A.2.12., Theorem 6.2.1]{Fukushimaetal}. 
\begin{remark}
Note that due to the refinement obtained obtained by Proposition~\ref{thm:1} for
the reflecting diffusion process $\X$, the boundary process
$\widehat{\X}$ can be defined without exceptional set, i.e., for every
starting point $x\in\partial D$.
\end{remark}

We note that the representation Theorem~\ref{thm:cont} can be expressed
with the help of the boundary process $\vX$, the first exit time
$\tau_D$ and the first exit place $X_{\tau_D}$.
\begin{lemma}
  \label{lem:change}
  We have that
  \begin{equation}\label{CVF}
   \int_{\tau(a)}^{\tau(b)} f(s) \dd  L_s = \int_a^b f(\tau(s)) \dd  s 
  \end{equation}
  for every bounded and measurable function $f \colon \R \to \R$.
\end{lemma}
\begin{proof}
  This follows by Monotone Class Theorem from the observation that
  $$[L_a,L_b] = \tau \circ g_{a,b},$$ where we have set $g_{a,b}(t):= [ t \in [a,b] ]$.
\end{proof}

\subsection{Infinitesimal generator}\label{section:infgen}
The following theorem yields the infinitesimal generator of the boundary process.
\begin{theorem}\label{thm:gen}
Let $\kappa:\overline{D}\rightarrow \R^{d\times d}$ be a symmetric,
uniformly bounded and uniformly elliptic conductivity.
Then the infinitesimal generator of
$\widehat{\X}$ is the Dirichlet-to-Neumann map
$\Lambda_{\kappa}$.
\end{theorem}
\begin{proof}
We will give two different proofs for this result. One is with the trace
Dirichlet forms and one with the help of Theorem~\ref{thm:cont} and
Change of Variables lemma~\ref{lem:change}.

Let $v, w \in \D(\widehat\E) = H^{1/2}(\partial D)$. Therefore,
\[
\begin{split}
\widehat\E (v,w) = \E(\mathcal Hv, \mathcal Hw)
&= \int_{\partial D} w(x) (\partial_{\kappa\nu(x)} \mathcal
Hv(x))\vert_{\partial D}
\dd\sigma(x) \\
&= \int_{\partial D} w(x) (\partial_{\kappa\nu(x)} \mathcal
Hv(x))\vert_{\partial D}
\dd\sigma(x) \\
&= \langle w, \Lambda_{\kappa} v\rangle_{H^{1/2}(\partial D),H^{-1/2}(\partial D)}
\end{split}
\]
We can factorize $\Lambda_{\kappa} = A^*A$, where $A \colon
H^{1/2}(\partial D) \to
L^2(\partial D)$. The $A$ can be seen as an unbounded operator in
$L^2(\partial D)$ with domain $\D(\widehat\E)$. Thus,
\[
\widehat\E (v,w) = \langle Aw, Av \rangle_{L^2(\partial D)},
\]
which implies  $A = \sqrt{-\mathcal L}$ or $\Lambda_\kappa = -\mathcal
L$, where $\mathcal L$ is the infinitesimal generator of $\widehat{\X}$
(c.f.~\cite{Fukushimaetal}).

The another proof is even more probabilistic in nature. Let $\varphi \in
C^\infty(\partial D)$ and $v = \Lambda_\kappa \varphi$. We know that $v$
is H\"older continuous and $\langle v, 1\rangle_{\partial D} = 0$.
Therefore, we may define
\[
g_R(x) := \EW_x A_R := \EW_x \int_0^R v(\vX_s) \dd s \quad
\text{and} \quad g_\infty(x) = \lim_{R \to \infty} g_R(x).
\]
Note that the Theorem~\ref{thm:cont} together with the Change of
Variables Lemma~\ref{lem:change} gives $g_\infty =
\Lambda^{-1}_\kappa v = \varphi$ and in particular, $g_\infty \in
C^\infty(\partial D) \subset \D(\mathcal L)$.
The transition operator $\widehat T_t$ applied to $g$ gives by Markov
property 
\[
\widehat T_t g_R(x) = \EW_x \EW_{\widehat X_t} A_R = \EW_x (A_{R+t} -
A_t).
\]
Therefore,
\[
(\widehat T_t - I) g_R(x) = -\EW_x A_t + \EW_x (A_{R+t} - A_R)
\]
and the Theorem~\ref{thm:cont} together with dominated convergence
theorem imply that
\[
(\widehat T_t - I) \varphi(x) = -\EW_x A_t = -\int_0^t \widehat T_sv(x)
\dd s.
\]
This in turn implies
\[
\mathcal L \varphi(x) = -v(x) = -\Lambda_\kappa \varphi(x)
\]
which implies the claim since $(\widehat\E, \D(\widehat\E))$ is regular.
\end{proof}

As in \cite{Hsu2}, we show that the Dirichlet-to-Neumann map is an
\emph{integro-differential operator }in the sense of Lepeltier and
Marchal~\cite{LepeltierM}. 

\begin{theorem}
  \label{lemma:integrodiff}
Let $\kappa:\overline{D}\rightarrow\R^{d\times d}$ be a symmetric, uniformly bounded and uniformly elliptic conductivity with components $\kappa_{ij}\in C^{0,1}(\overline{D})$, $i,j=1,...,d$ such that $\kappa$ satisfies (A1). 
Then the Dirichlet-to-Neumann map $\Lambda_\kappa$ is of form
  \[
  \Lambda_\kappa \phi =  \Lambda_\kappa \mathrm{id} \cdot \nabla_T \phi + A_{\kappa}\phi\quad\text{for all }\phi\in H^{1/2}(\partial D),
  \]
where $A_{\kappa}$ is given by the integral operator
  \[
  A_{\kappa} \phi(x): = \int_{\partial D} ( \phi(y)-\phi(x) - \nabla_T \phi(x) \cdot
  (y-x)) N_{\kappa}(x,y) \dd \sigma(y) 
  \]
   for a.e. $x\in\partial D$ and 
  \begin{equation*}
  N_{\kappa}(x,y):=\partial_{\kappa\nu(x)}K_{\kappa}(x,y),\quad x,y\in\partial D.
  \end{equation*}
\end{theorem}

\begin{proof}
By a density argument, we may assume that $\phi \in C^{\infty}(\partial
D) \cap H^{1/2}(\partial D)$. The unique solution in $H^1(D)$
of the Dirichlet problem (\ref{eqn:Dirichletprob}) is by Lemma~\ref{Poissonkernel}
\[
\mathcal{H}\phi(x) = \EW_x \phi(\X_{\tau_D}) = \int_{\partial D} 
 \int_{0}^{\infty}\phi(y)\partial_{\kappa\nu(y)} p^0(t,x,y)\dd  t\dd
 \sigma(y).
\]
Therefore, $\Lambda_\kappa$ maps $\phi$ to 
\begin{equation}
  \label{equ:DN1}
  \Lambda_\kappa \phi(x) = \partial_{\kappa\nu(x)} \mathcal H\phi(x).
\end{equation}
Let us extend $\phi$ into the neighborhood of the boundary as constant
along the conormal direction. Therefore, the first order tangential
derivative $V := \nabla_T \phi$ gets extended simultaneously.
We will denote the extensions $\widetilde \phi$ and $\widetilde
V$, respectively. We compute the conormal derivative of the function
\[
w: = \mathcal H \phi - \widetilde \phi \mathcal H 1 -
\sum_{j=1}^d \widetilde
{V_j} \big(\mathcal H W_j - \widetilde W_j \mathcal H 1\big),
\]
where $\{W\}$ is a vector field on the boundary defined by
$W (y): = y_T$ as the projection to the tangent plane going through the
point $y$.
By construction, the conormal derivative commutes with
multiplication by the extended functions and vector fields. Therefore,
\[
\partial_{\kappa\nu} w = \Lambda_\kappa \phi - \phi \Lambda_\kappa 1 - 
\sum_{j=1}^d V_j \Lambda_{\kappa} W_j
= \Lambda_\kappa \phi - V \cdot \Lambda_{\kappa} W,
\]
where $\Lambda_\kappa 1 = 0$ since $u(x) = 1$ in $\overline D$
is the unique solution to the corresponding Dirichlet problem and therefore, the
conormal derivative vanishes identically.
As we have shown the existence of a Poisson kernel, we can write the left-hand side in a different way, namely 
\[
\nabla w(x) = \nabla_x \int_{\partial D} \big(\phi(y)-\widetilde \phi(x)-
\widetilde V(x)\cdot (\widetilde W(y) - \widetilde W(x))\big) K_{\kappa}(x,y) \dd \sigma(y)
\]
for almost every $x$ in a neighborhood of the boundary.

Next, we show that the function $y \mapsto N_{\kappa}(x,y)
(\lvert y-x\rvert^2\wedge 1)$ is integrable with respect to the surface measure
$\sigma$ for every $x \in \partial D$.
When $\kappa \equiv 1$, the proof given in \cite{Hsu2} yields the claim
for the kernel $N_1$ corresponding to $\kappa \equiv 1$. As we assume
that $\kappa$ satisfies (A1), the operator $\Lambda_\kappa -
\Lambda_{1}$ is a smoothing operator which follows by the standard
elliptic regularity, cf.~\cite{HankeHR}. This implies that the kernels
$N_1$ and $N$ have the same leading singularities and the claim thus
follows from the estimate for $N_1$. As a consequence, we can use the
dominated convergence theorem
to take the differentiation inside the integration and we obtain 
\[
\partial_{\kappa\nu} w(x) = \int_{\partial D} \big(\phi(y)- \phi(x)-
\nabla_T\phi(x)\cdot (y - x)\big) N_{\kappa}(x,y) \dd  \sigma(y) = A_{\kappa}\phi(x)
\]
for a.e. $x \in \partial D$.
\end{proof}

\section{Novel probabilistic formulations of the Calder\'{o}n problem}\label{section:probinv}
By the considerations from above, the boundary process is uniquely
determined by the absorbing diffusion process $\X^0$ so that Theorem
\ref{thm:gen} leads to the following probabilistic interpretation of
Calder\'{o}n's problem: 
\emph{Given a boundary process $\widehat{\X}$ associated with the
regular $\sigma$-symmetric Dirichlet form
$(\widehat{\E},H^{1/2}(\partial D))$, is $\X^0$ the unique absorbing
diffusion process on $D$ such that $\widehat{\X}$ is the boundary
process of the corresponding reflecting diffusion process $\X$ on
$\overline{D}$?}

The Calder\'on problem in $2$ dimensions is known to be solvable for 
isotropic $\kappa \in L^{\infty}(D)$. Given the boundary
process $\vX = \vX_{\kappa_0}$ we can thus uniquely determine the generator
$\Lambda = \Lambda_{\kappa_0}$. The celebrated result of
Astala and P\"aiv\"arinta~\cite{AstalaP} says that whenever $\Lambda_\kappa =
\Lambda$ and both $\kappa$ and $\kappa_0$ are isotropic, 
uniformly bounded and uniformly elliptic, then $\kappa = \kappa_0$.
Therefore, the equality $\X_\kappa = \X_{\kappa_0}$ must hold as well.

The recent result by Haberman and Tataru~\cite{HabermanTataru} implies
the same for three and higher when $\kappa$ and $\kappa_0$ are
assumed to be $C^1(D)$ or if they are Lipschitz continuous and close to identity
in certain sense. In recent preprints, Haberman \cite{Haberman} improved this to the even $W^{1,n}$
conductivities for the dimensions three and four, while Caro and Rogers~\cite{CaroRogers} 
improved the Haberman and Tataru technique to prove the uniqueness
for the Lipschitz case in general in any dimension.

When the conductivity is not assumed to be isotropic the uniqueness has
always an obstruction, namely we have $\Lambda_{\kappa_0}=
\Lambda_{\kappa_1}$, whenever $\kappa_1 = F_* \kappa_0$ is the
push-forward conductivity by a diffeomorphism $F$ on $D$ that leaves the
boundary $\partial D$ invariant. In the plane, this is known to be the
only obstruction by the result of Astala, Lassas and
P\"aiv\"arinta~\cite{AstalaLP} which holds without additional regularity assumptions
on the conductivity. In higher dimensions the question is still very
much open in general, see~\cite{DSFKSU} for further discussion.

It should be emphasized that these results from analysis all rely on
so-called \emph{complex geometric optics solutions} and the authors are
not aware of any probabilistic interpretation of such solutions.
Therefore, it is an open problem to find a probabilistic solution to the probabilistic formulation of Calder\'{o}n's problem.

Let us elaborate a bit on this problem by providing three seemingly different but equivalent versions. 
\subsection{First version}
It is an immediate
consequence of Theorem \ref{lemma:integrodiff} that
the jumping measure of $\vX$ can be described with the help
of~\cite[Th{\'e}or{\`e}me 10]{LepeltierM}.

\begin{lemma}\label{lem:levy}
Let $\kappa:\overline{D}\rightarrow\R^{d\times d}$ be a symmetric,
uniformly bounded and uniformly elliptic conductivity with components
$\kappa_{ij}\in C^{0,1}(\overline{D})$, $i,j=1,...,d$, such that
$\kappa$ satisfies (A1). 
Then for every non-negative Borel function $\phi:\partial
D\times\partial D\rightarrow\R_+$, vanishing on the diagonal, and any
stopping time $\hat{\tau}$ of $\widehat{\X}$, we have
\begin{equation}
   \EW_x \sum_{s \le \hat{\tau}} \phi(\vX_{s-}, \vX_s)[\vX_{s-} \ne
   \vX_s] = \EW_x \int_0^{\hat{\tau}}
   \int_{\partial D} \phi(\vX_s, y) N_{\kappa}(\vX_s, y)\dd \sigma(y)\dd  s.
\end{equation}
\end{lemma}
\begin{proof}
  Suppose first that $\textrm{diam}(D) \le 1$. We note
  that the operator $A_{\kappa}$ in Theorem~\ref{lemma:integrodiff} coincides with
  the integral operator causing the jumps. If $\psi \in
  C^{\infty}(\partial D)\cap H^{1/2}(\partial D)$ and it is continued as
  $\tilde{\psi} \in H^2(\R^d)$ so that $\psi$ and its
  tangential derivative are continued as constants along the conormal
  directions in the neighborhood of the boundary $\partial D$, then for
  every $x \in \partial D$, we have
  \[
  A_{\kappa}\psi (x) = \int_{\R^d \setminus \{0\}}
  \big(\tilde{\psi}(x+z) - \tilde{\psi}(x) - \big[\lvert z\rvert\le
  1\big] z \cdot \nabla \tilde{\psi}(x) \big)
  S_{\kappa}(x,\mathrm{d} z),
  \]
  where we have set $S_{\kappa}(x,\mathrm{d} z):= N_{\kappa}(x,x+z)\dd  \sigma_x(z)$ and
  $\sigma_x(B):= \sigma(x+B)$ for every Borel set $B \subset \R^d$. 
  In the same way we can extend the drift term so that the corresponding
  integro-differential operator is the infinitesimal generator of an
  extended process $\overline{\X}$ on the whole
  space $\R^d$.  Since we know that $\vX_t \in \partial
  D$ for all $t\geq 0$, it follows that the extended process $\overline{\X}$ will stay on the boundary
  $\partial D$ if we start it from the boundary and it coincides with
  $\vX$ there.
  
  For this extended process $\overline{\X}$ we can apply
  the result~\cite[Th{\'e}or{\`e}me 10]{LepeltierM}
  and we obtain
  \[
  \begin{split}
  &\EW_x \sum_{s \le \hat{\tau}} \phi(\vX_{s-}, \vX_s)[\vX_{s-} \ne \vX_s] \\
  &= \EW_x \int_0^{\hat{\tau}} 
  \int_{\R^d\setminus \{0\}} \phi(\vX_s, \vX_s + y)
  N_{\kappa}(\vX_s, \vX_s + y) \dd  \sigma_{\vX_s}(y)  \dd  s
  \end{split}
  \]
  for any non-negative Borel function $\phi:\partial D\times\partial
  D\rightarrow\R_+$ vanishing on the diagonal and any stopping time
  $\hat{\tau}$ of $\widehat{\X}$.
  The claim follows now in this special case by change of integration variable.

  The general case follows by scaling: Let us denote $\X^R_t :=
  R^{-1}\X_t$. This is a reflecting diffusion process corresponding to $\kappa^R$ on
  a domain $D^R$, where $D^R := R^{-1}D$ and $\kappa^R(x) :=
  R^{-2}\kappa(Rx)$. Since the diameter of $D^R$ is one, the claim holds
  for the boundary process $\vX^R$ of $\X^R$.
  
  Let $L^R$ denote the local time of $\X^R$ on the boundary $\partial
  D^R$. By definition, this is in Revuz correspondence with the surface
  measure $\sigma^R$ of the boundary $\partial D^R$. By using the Revuz
  correspondence and change of
  variables, it follows that 
  \[
  L_t^R = RL_t.
  \]

  Therefore, the right-inverse $\tau^R$ of the local time $L^R$ has a
  scaling law 
  \[
  \tau^R(t) = \tau(R^{-1} t)
  \]
  which in turn implies that the boundary processes scale by the law
  \[
  \vX^R_t = R^{-1}\vX_{R^{-1}t}
  \]
  and that $\eta$ is an $\vX$-stopping time if and only if $R\eta$ is an
  $\vX^R$-stopping time.

  If we denote by $N^R_{\kappa}$ the conormal derivative of the Poisson kernel of
  $\X^R$ multiplied by $2$ and compute the scaling law, we find out that
  \[
  N^R_{\kappa}(x,y)= R^{d-2} N_{\kappa}(Rx,Ry).
  \]
  With all these scaling laws, we are now ready to prove the claim for $\vX$. Let
  $\phi^R(x,y) := \phi(x,y)$ and let $\eta$ be an $\vX$-stopping time. We have
  \[
  \EW_x \sum_{s \le \eta} \phi(\vX_s, \vX_{s-})[ \vX_s \ne \vX_{s-}]
  =
  \widehat{\EW}_{R^{-1}x} \sum_{s \le \eta R} \phi^R(\vX^R_s, \vX^R_{s-})[ \vX^R_s
  \ne \vX^R_{s-}],
  \]
  where $\widehat{\EW}_{R^{-1}x}$ denotes the expectation given $\vX^R_0
  = R^{-1} x$.
  By the first part of the proof, the right-hand side is equal to
  \[
  \widehat{\EW}_{R^{-1}x} \int_0^{\eta R} \int_{\partial {D^R}} \phi^R(\vX^R_s,
  y) N^R_{\kappa}(\vX^R_s, y)\dd \sigma(y)\dd  s.
  \]
  With the change of variables $y' = R y$ and $s' = s R^{-1}$ and the
  scaling law $N_{\kappa}^R(\vX^R_{Rs}, R^{-1} y) = R^{d-2} N_{\kappa}(\vX_s, y)$, the
  claim follows.
\end{proof}
This result states that the pair $(N_{\kappa}(x,y)\dd\sigma(y), \mathrm {id}_t)$ is a
\emph{L\'evy system }of the Hunt
process $\vX$. 
Since the positive continuous additive functional $\mathrm {id}_t = t$ with respect to $\vX$ has the Revuz
measure $\sigma$, we obtain from \cite[Theorem A.1 (iii)]{Chen2} that $\frac12 N_{\kappa}(x,y)\dd\sigma(y)\dd \sigma(x)$ coincides with the jumping measure $\dd\widehat{\mathcal{J}}(x,y)$ of the boundary process $\vX$. In particular, the L\'evy system characterizes the boundary process completely and we have obtained another probabilistic formulation of Calder\'{o}n's problem: \emph{Given pure jump processes on $\partial D$ generated by $\Lambda_{\kappa_1}$ and $\Lambda_{\kappa_2}$, respectively. Show that the corresponding L\'{e}vy systems coincide, i.e., $N_{\kappa_1}\equiv N_{\kappa_2}$, if and only if $\kappa_1\equiv\kappa_2$.}

\subsection{Second version}
An interesting aspect with regard to the first version is the fact that it may be translated into an equivalent parabolic \emph{unique continuation problem }in terms of transition kernel densities of absorbing diffusion processes on $D$. More precisely, let $\kappa_1, \kappa_2 \in C^{0,1}(\overline{D})$ be isotropic conductivities and $p_1^0$, $p_2^0$ the corresponding transition kernel densities. Define for any function $\phi\in L^2(D)$ the functions $$v(t,x):=\int_D(p^0_1(t,x,y)-p^0_2(t,x,y))\phi(y)\dd y,\quad w(t,x):= \int_Dp^0_2(t,x,y)\phi(y)\dd y$$ and set $$c(t,x):=-\nabla\cdot(\kappa_1\nabla w(t,x))+\nabla\cdot(\kappa_2\nabla w(t,x)).$$ By these definitions, we have
\begin{equation}
\begin{cases}
\partial_t v(t,x)=\nabla\cdot(\kappa_1\nabla v(t,x))+c(t,x),\quad&(t,x)\in(0,\infty)\times D\\
v(0,x)=0,\quad& x\in D\\
v(t,x)=0,\quad& (t,x)\in(0,\infty)\times\partial D.
\end{cases}
\end{equation}
By the probabilistic formulation from above, 
the uniqueness result from Caro and Rogers~\cite{CaroRogers} corresponds to the following assertion: \emph{Assume that 
\begin{equation}
\int_0^{\infty}\partial_{\kappa\nu(x)}\partial_{\kappa\nu(y)}(p_1^0(t,x,y)-p^0_2(t,x,y))\dd t\equiv 0.
\end{equation}
Then $v$ must be identically zero for any $\phi\in L^2(D)$.}

\subsection{Third version}
Let us conclude this section by deriving yet another equivalent formulation of Calder\'{o}n's problem which is based on the construction for the reflecting Brownian motion from \cite{Hsu2}. To be precise, as in \cite{Hsu2}, we adopt It\^{o}'s idea of regarding the excursions of $\X$ from the boundary as a point process taking values in the space of excursions. 
First, we recall some definitions from \cite{Hsu2}. A measurable function $e:[0,\infty)\rightarrow \partial D\cup\{\partial\}$ is called a \emph{point function} if the set $J(e):=\{s>0:e_s\in\partial D\}$ is countable and we denote the set of point functions by $\mathcal{P}(\partial D)$. For each point  function $e$, the \emph{counting measure} on $(0,\infty)\times\partial D$ is defined by 
\begin{equation}
n_e(B):=\{(s,x)\in B:e_s=x\},\quad B\subset(0,\infty)\times\partial{D}.
\end{equation}

Given the probability space $(\widehat{\Omega}, \widehat{\F},\widehat{\mathbb{P}}_x)$ with filtration $\{\widehat{\F}_t, t\geq 0\}$, a function $e:\widehat{\Omega}\rightarrow\mathcal{P}(\partial D)$ is called a \emph{point process} if for each $B\subset\B(\partial D)$, the increasing process $t\mapsto n_e((0,t]\times B)$ is adapted to $\{\widehat{\F}_t,t\geq 0\}$.

Let $W^{a,b}$ denote the space of continuous \emph{excursions} of $\X$ from $a\in\partial D$ to $b\in\partial D$, i.e., the space of continuous paths $e\in C([0,\infty);\overline{D})$ such that $e(0)=a$ and there exists an $l>0$ such that $e(t)\in D$ for all $0<t<l$ and $e(t)=b$ for all $t\geq l$. If the space $W^{a,b}$ is equipped with the filtration $\sigma(e(s),e\in W^{a,b},0\leq s\leq t)$,
then the space of all excursions is given by 
\begin{equation}
W=\bigcup_{a\in\partial D}W^a, 
\end{equation}
where 
\begin{equation}
W^a:=\bigcup_{b\in\partial D, b\neq a}W^{a,b}.
\end{equation}

The same proof as in~\cite[Proposition 4.4]{Hsu2} shows that the random
set of jump times $\{ \vX_{s-} \ne \vX_s\}$ is a countable and dense set
and that there is a constant $c > 0$, depending only on the domain $D$, such that
after any given time $t \geq 0$, there are always infinitely many jumps
of size at least $c$. Therefore, we set 
$$J:=\{s\in (0,\infty):\tau(s-)<\tau(s)\}\quad\text{and}\quad l(s):=\tau(s)-\tau(s-)$$
and define the \emph{point process of excursions }of $\X$ as a $W$-valued point process such that
\begin{equation}
e_s(t):=\begin{cases}
\X_{t+\tau(s-)},\quad &t\leq l(s)\\
\X_{\tau(s)},\quad&t > l(s),
\end{cases}
 \end{equation}
 if $s\in J$, whereas $e_s:=\partial$ if $s\not\in J$. 
 
Clearly, the family $\{e_t,t\leq 0\}$ is adapted to $\{\widehat{\F}_{t},t\geq 0\}$. The following definition yields a useful characterization of point processes.
\begin{definition} 
A $\sigma$-finite random measure $\hat{n}_e$ on the measurable space $((0,\infty)\times\partial D,\B((0,\infty))\times \B(\partial D))$ is called the \emph{compensating measure }of a point process $e$ if there is a sequence $\{B_k\}_{k\in\N}\subset\B(\partial D)$, exhausting $\partial D$, such that 
\begin{enumerate}
\item[(i)] the mapping $t\mapsto \hat{n}_e((0,t)\times B_k)$ is continuous for every $k\in \N$;
\item[(ii)] $\EW \hat{n}_e((0,t)\times B_k)<\infty$ for every $k\in\N$;
\item[(iii)] for every $B\in\B(\partial D)$ contained in some $B_k$, the process
\begin{equation*}
t\mapsto n_e((0,t)\times B)-\hat{n}_e((0,t)\times B)
\end{equation*}
is a martingale with respect to $\{\widehat{\F}_{t},t\geq 0\}$.
\end{enumerate}
\end{definition}
The following theorem is a straightforward generalization of \cite[Theorem 5.1]{Hsu} for the reflecting Brownian motion. It
provides a \emph{compatibility condition}, connecting the excursion law $\mathbb{P}_{a,b}$ on $W^{a,b}$, which is completely determined by $p^0$, with the L\'{e}vy system of $\widehat{\X}$. 

\begin{theorem}
Let $\kappa:\overline{D}\rightarrow\R^{d\times d}$ be a symmetric, uniformly bounded and uniformly elliptic conductivity with components $\kappa_{ij}\in C^{0,1}(\overline{D})$, $i,j=1,...,d$, such that $\kappa$ satisfies (A1). 
Then the point process of excursions of $\X$ admits the compensating measure 
\begin{equation}\label{eqn:compensating}
\hat{n}_e((0,t)\times B)=\int_0^t\mathbb{Q}_{\widehat{\X}_s}\{B\cap\{e:e(0)=\widehat{\X}_s\}\}\dd s, 
\end{equation}
where the \emph{excursion law }$\mathbb{Q}_a$ from $a\in\partial D$ is a $\sigma$-finite measure on $W^a$ defined by
\begin{equation}\label{eqn:excursion}
\mathbb{Q}_{a}\{B\}=\int_{\partial D}\mathbb{P}_{a,b}\{B\cap\{e:e(l)=b\}\}N_{\kappa}(a,b)\dd\sigma(b)
\end{equation}
for every measurable $B\subset W^a$. \end{theorem}

We have thus arrived at the following equivalent probabilistic formulation of Calder\'{o}n's problem: 
\emph{Given the boundary process $\widehat{\X}$ with its corresponding L\'{e}vy system  $(N_{\kappa}(x,y)\dd\sigma(y), \mathrm {id}_t)$, show that $\mathbb{P}_{a,b}$ determined by $p^0$ is the unique excursion law such that (\ref{eqn:compensating}) yields the compensating measure of the point process of excursions of some reflecting diffusion process on $\overline{D}$.} 

\section{Conclusion and outlook}\label{section6}
In this work, we have obtained a probabilistic formulation of Calder\'{o}n's inverse conductivity problem. 
This formulation comes in three equivalent versions and each of them may yield both
a novel perspective as well as a novel set of (probabilistic) tools when it comes to studying questions related to the unique determinability of conductivities from boundary data.
Indeed, it was shown in~\cite{Hsu2} that for the case $\kappa \equiv 1/2$ and under a certain consistency assumption, the reflecting Brownian motion can be reconstructed from its point process of excursions and the boundary process. However, the consistency assumption was derived by using the reflecting Brownian motion to begin with and as it is noted in~\cite{Hsu2}, there might 
be other consistency assumptions leading to other constructions. Showing that there are no other consistent constructions is equivalent to the probabilistic inverse problem which will be the subject of future research.

\section*{Acknowledgments} 
The research of M.~Simon was supported by the Deutsche Forschungsgemeinschaft (DFG) under grant HA 2121/8 -1 58306.
The research of P.~Piiroinen was supported by Academy of Finland (AF)
under Finnish Centre of Excellence in Inverse Problems Research
2012--2017, decision number 250215. He has also been supported by an AF
project, decision number 141075 and ERC grant 267700, InvProb. The
authors are pleased to thank Martin Hanke for insightful comments on
this work.


\end{document}